\documentclass[reqno]{amsart}
\usepackage{amsmath}
\usepackage{mathtools}

\usepackage{amsthm}
\usepackage{tikz}
\usetikzlibrary{arrows}

\usepackage[colorlinks=true]{hyperref}
\usepackage[all]{xypic}
\usepackage{graphicx}
\newtheorem{theorem}{Theorem}
\newtheorem{proposition}{Proposition}
\newtheorem{corollary}{Corollary}

\theoremstyle{definition}
\newtheorem{definition}{Definition}

\theoremstyle{remark}
\newtheorem{remark}{Remark}

\newtheorem{example}{Example}

\newcommand{\defin}[1]{{\it #1}}
\newcommand{\R}{\mathbb{R}}
\newcommand{\N}{\mathbb{N}}

\newcommand{\Q}{\mathbb{Q}}

\newcommand{\C}{\mathcal{C}}
\newcommand{\A}{\mathcal{A}}

\newcommand{\uni}{\mathfrak{m}}
\newcommand{\Leb}{\lambda}

\makeatletter
\renewcommand\section{\@startsection
	{section}% nom du titre
{1}% niveau de titre
{0pt}% indentation
{-3.5ex plus -1ex minus -.2ex}%espace vertical avant 
{2.3ex plus.2ex}% espace vertical après 
{\centering\normalfont\Large\scshape}}

\renewcommand\subsection{\@startsection
	{subsection}% nom du titre
{2}% niveau de titre
{0pt}% indentation
{-3ex plus -1ex minus -.2ex}%espace vertical avant 
{1ex plus.2ex}% espace vertical après 
{\normalfont\large\bfseries}}

\renewcommand\subsubsection{\@startsection
	{subsubsection}% nom du titre
{3}% niveau de titre
{0pt}% indentation
{-1.5ex plus -1ex minus -.2ex}%espace vertical avant 
{0.8ex plus .2ex}% espace vertical après 
{\normalfont\bfseries}}

\renewcommand\paragraph{\@startsection
	{paragraph}% nom du titre
{4}% niveau de titre
{0em}% indentation
{-1.2ex plus -0.4ex minus -.2ex}%espace vertical avant 
{0\baselineskip}% espace vertical après 
{\normalfont\bfseries}}

\makeatother

\newcounter{claimcount}
\newcommand{\THMfont}[1]{{\sl #1}}
\newcommand{\Claim}[1]{\refstepcounter{claimcount} \vspace{0.3em}               \noindent {\sc Claim \theclaimcount: \ }\THMfont{ #1}}
\newcommand{\bprf}[1][Proof:]{\begin{list}{}    {\setlength{\leftmargin}{0.5em} \setlength{\rightmargin}{0em}  \setlength{\listparindent}{1em}}   \item {\em \hspace{-0.8em}  #1  }}
\newcommand{\eprf}{\end{list}}
\newcommand{\bclaimprf}{\bprf}
\newcommand{\eclaimprf}{ \hfill $\Diamond$~{\scriptsize {\tt Claim~\theclaimcount}}\eprf} % %

\title[Algorithmic complexity of attractors]{On the algorithmic descriptive complexity of attractors in topological dynamics}

\author{Crist\'obal Rojas}
\address{Crist\'obal Rojas, Instituto de Ingenier\'ia Matem\'atica y Computacional, Universidad Cat\'olica de Chile}

\author{Mathieu Sablik}
\address{Mathieu Sablik, IMT, Université Toulouse III - Paul Sabatier, Toulouse, France}
\thanks{Crist\'obal Rojas was partially supported by Grants ANID/FONDECYT Regular 1230469 and ANID/Basal National Center for Artificial Intelligence CENIA FB210017. Mathieu Sablik was partially supported by ANR project Difference (ANR-20-CE48-0002) and the project Computability of asymptotic properties of dynamical systems from CIMI Labex (ANR-11-LABX-0040).}

\begin{document}
\maketitle

\begin{abstract}

We study the computational problem of rigorously describing the asymptotic behaviour of topological dynamical systems up to a finite but arbitrarily small pre-specified error. More precisely, we consider the limit set of a typical orbit, both as a spatial object (attractor set) and as a statistical distribution (physical measure), and prove  upper bounds on the computational resources of computing descriptions of these objects with arbitrary accuracy. We also study how these bounds are affected by different dynamical constrains and provide several examples showing that our bounds are sharp in general. In particular, we exhibit a computable interval map having a  unique transitive attractor with Cantor set structure supporting a unique physical measure such that both the attractor and the measure are non computable.  
\end{abstract}

%\tableofcontents

\section{Introduction}

Dynamical Systems are abstract mathematical models of systems whose state evolves according to a prescribed rule. They are often used to model natural phenomena, and have an enormous range of applications. This versatility, however, makes them notoriously hard to analyse in general. The last decades have seen an increasing use of computers in the study and analysis of dynamical systems, which have resulted in a number of theoretical breakthroughs; a notable example is the role computers played in both the discovery of the Lorenz attractors through numerical simulations \cite{Lor63}, and the mathematical proof (40 years later) of their existence \cite{Tuc02} in which computers were used to assist the proof by verifying certain quantitative conditions to hold through certified computations. In fact, the use of computers to explore the properties of dynamical systems so as to make conjectures about their behaviour and even to find ideas of how to prove those conjectures, is becoming increasingly popular among working mathematicians.

While there exists countless papers in the literature about computations in dynamical systems, only a small fraction of them addresses  the problem rigorously; i.e., how far is the sought actual quantity from the computed one? And can a computer really perform such computation up to a very small pre-specified error? 

An important example where these questions are of interest is given by the \textit{reachability} problem: given a finite description of a system, can we use a computer-aided procedure to guarantee that if the system starts at a given initial state, its evolution will never reach some ``unsafe region'' \cite{AsMalAm95,Col05} ?  

Several results have been obtained in this direction. In general, while bounded-time simulations are usually possible, the long-term behavior features of many of the interesting systems may be very hard to compute. 

The idea here is that, since rich enough physical systems can, in theory, perform universal computation in the sense that they can simulate any Turing machine \cite{Neumann66,Minsky:67,Sontag:91,graca2021analytic,Cardona_2021,Cardona2_2021}, the computational unsolvability of problems like the Halting Problem \cite{Turing1936} entails the computational unpredictability of this kind of systems --- most of their long term properties are non-computable \cite{Mo91,Koiran:94,Reif:1994,AsMalAm95,BY06,BraYam07,GalHoyRoj07c}.  

On the other hand, it has also been observed that the computational capabilities of these systems may be affected by restricting some of the features related to their physical plausibility, such as  degrees of freedom, compactness, regularity or robustness --- the long term behaviour of such restricted systems  may be easier to predict \cite{Navot04, Braver12, PRL15, BRS17, Asa01, Cardona_23}.  

The mathematical theory of dynamical systems provides a landscape of dynamical properties that a system may exhibit in the long-term. At one end, there are systems with a rather ``ordered'' dynamics, in the sense that most nearby initial conditions remain close in time and give rise to individual trajectories whose common behavior is easy to describe, e.g. when most of initial conditions converge to the same fixed point (or periodic orbit). In particular, in this case the reachability problem typically becomes decidable, and it is usually possible to give a rigorous computational description of the limiting dynamics \cite{BGH10, Asa01}.

At the other end, we have chaotic systems, where sensitivity to initial conditions makes it impossible to in practice simulate individual trajectories for an extended period of time. In this case, instead of attempting to describe individual trajectories, one should study the limit set of collections of typical trajectories -- both as a spatial and as a statistical invariant object \cite{Palis}. The structure of these \textit{attractors}  encapsulates the properties of the asymptotic behaviour of the underlying systems, and constitute therefore the objects that should be sought to be computed for their analysis and prediction.

A number of results have also been obtained in this direction, showing a variety of possibilities from the point of view of the algorithmic complexity of describing the limiting behavior of a typical orbit \cite{BY06,Hellouin-Sablik-2018,GRZ18,YRAdvances19,RY_2020}.

 In this paper we introduce a general framework based on computable analysis to study and classify dynamical systems according to the computational difficulty of computing arbitrarily good approximations of their attractors. While in general a dynamical system may have several (even infinitely many) different attractors, in this paper we choose to focus on the case of systems whose overall dynamical picture is ``simple" in the sense that they have a \emph{unique} well defined limiting behavior, which is followed by all ``typical" initial conditions. Our main goal is to obtain insight on the computational difficulty of the problem of describing this unique limiting behavior, and on how it depends on the dynamical properties of the system. 

 The case of systems having more than one typical asymptotic behavior is also interesting and we plan to study it in future work.  

\bigskip

\noindent\textbf{Structure of the paper and main results.} We begin with a discussion on the concept of attractor in Section~\ref{section.DefAttractor}, where the three versions to be studied -- topological, metric, and statistical attractors -- are introduced. The language we use to bound their algorithmic complexity is introduced in Section~\ref{section.ComputableAnalysis}. It is intended to capture the difficulty of computing a complete description of these attractors as closed sets. Modulo some technicalities, it essentially amounts to determining the complexity in the arithmetical hierarchy of the problem of deciding whether a given rational ball intersects the set. 

Section \ref{section.UpperBounds} is devoted to provide general upper bounds on the complexity of computing the attractor, for each of the three different types introduced. We also study the effect that different dynamical properties have on these bounds. While in general \emph{transitivity} does not affect them, other properties such as being \emph{strongly attracting, minimal or exact}, force the attractor to be strictly less complex. 

Section ~\ref{section.symbolic} contains the main contributions of this paper. It is devoted to show that the upper bounds we obtain in Section \ref{section.UpperBounds} are tight. All the examples in this section are within the class of computable symbolic systems. Their symbolic nature makes them suitable for constructing examples of systems whose attractors are maximally complex, but whose overall dynamics is still possible to understand, thus providing insight into how algorithmic complexity is built into the attractor by the evolution of the system's dynamics. Our examples show that for each type of attractor, there are computable systems with an attractor of the corresponding type which is \emph{complete} for the upper bounds given in the previous section. While in many natural cases these three types of attractors coincide (and have therefore the same complexity), in general they can be different. These are known as \emph{wild} attractors. In Theorem~\ref{Corollary.DifferentAttractors}, we show that systems with wild attractors can have extremely different algorithmic complexities. In particular, we construct examples of systems whose \emph{metric} (respectively, topological) attractor is maximally complex ($\Pi_2$-complete), whereas its \emph{topological} (respectively, metric) attractor is computationally extremely simple. 

Finally, in Section \ref{interval} we explore this question for computable transformations of the unit interval. It is natural to expect that the more restricted topological nature of these maps might impose additional computability constrains on their attractors, as is for instance the case for topological entropy \cite{topentr}. However, at least for metric attractors, we show that this is not the case by exhibiting computable interval maps with different dynamical properties with a metric attractor matching the corresponding upper bound.  

\medskip
\noindent\textbf{Related work.} These results raise a lot of interesting questions. For example, to study whether our upper bounds are tight for more restricted classes of systems. Inspired by a preliminary version of the present work, this question has been studied for the class of cellular automata~\cite{Esnay-Ninez-Torma}, where several examples are provided. Another interesting direction is to study classes of systems that are restricted by analytic properties, such as their degree of regularity. In the unit interval for example, it was shown in \cite{YRAdvances19} that attractors of maps from the logistic family are always computable sets, but the problem is largely open for more general smooth maps. Other dynamical properties, such as some form of hyperbolicity, should also make attractors easier to compute. The geometric Lorenz attractor for example was shown to be computable in \cite{GRZ18} by exploiting this property. More generally, according to the \emph{physical} Church-Turing thesis, anything that can be computed by a physical device (however it may be constructed), can also be computed by a Turing Machine. Therefore, ``physically plausible" systems should have computable attractors. Thus, a very interesting challenge is to find a set of dynamical conditions that make the attractor computable, and that is applicable to a large set of systems. Finally, an important related question, that we do not explore here, is to characterize the algorithmic complexity of attractors for \emph{generic} maps within a given family of dynamical systems.

% !TEX root = Attractors.tex
\section{Preliminary definitions}\label{section.Def}

\subsection{On the concept of Attractor}\label{section.DefAttractor}

Let $T:X\to X$ be a continuous transformation over a compact set $X$ on which a natural reference measure $\Leb$ can be defined, that we will call Lesbesgue measure. The idea of the notion of attractor is simple:  they are invariant sets supporting the asymptotic dynamics of typical orbits. Although the intuition behind this idea is quite clear, it is not so easy to provide a precise definition and, in fact, several different definitions exists in the literature--we refer the reader to \cite{Milnor-attractor}, where an extended discussion of these issues is presented. In all cases, it seems that  any reasonable definition of attractor should at least include the following two points:

%\todo{Mettre à jour les r\'ef\'errances}

\begin{enumerate}
\item  an attractor must describe the limiting behavior of a large set of initial conditions (called its basin of attraction); and
\item all parts of the attractor should play a role in describing the limiting behavior of the points in its basin of attraction. 
\end{enumerate}

While there are many examples of systems possessing several (even infinitely many) different attractors, we see this phenomenon as a strong source of additional complexity in the description of the limiting behaviour of the system, which we would like to avoid in the analysis we do in this work. Therefore, we choose to focus our study on systems which are ``well behaved'' in the sense that they exhibit a \emph{unique typical limiting behavior} which is followed by most initial conditions. Our goal is to understand the computational difficulty of describing this well defined, unique limiting behavior.  

In formalising the two points above in our intuitive definition of attractor, the idea of ``large" will mean either \emph{generic}\footnote{A set is \emph{generic} if it contains a countable intersection of open and dense sets. The complement of a generic set is called \emph{meager}.} or \emph{full measure} with respect to some natural reference measure, that we will denote by $\lambda$. 

In addition, one can also consider (at least) two different ways in which the attractor may \emph{describe} the limiting behavior. The first possibility is to consider a topological description. Recall that $\omega(x)$ denotes the set of accumulation points of the sequence  $x, T(x), T^{2}(x), \dots$ of iterates of $x$ under  $T$.  For a closed invariant set $A\subset X$, the \defin{realm of topological attraction}\footnote{We do note require it to be an open set, which is why we do not call it the basin.} of $A$ is defined to be the set 
$$
\rho(A) = \{ x : \omega(x)\subset A\}. 
$$

%\todo{Rappeller que $\rho(A)$ est un $G_\delta$}

%Recall that a set is \defin{nowhere dense} if its closure has empty interior,  and \defin{meager} if it a countable union of nowhere dense sets. 
%Recall that a set is \defin{generic} if it contains a countable intersection of open and dense sets.
Following Milnor \cite{Milnor-attractor}, we define

\begin{definition}{\label{attractor_top}} A closed set $A\subset X$ will be  called a  \defin{metric attractor} if it satisfies two conditions:
\begin{itemize}
\item its topological basin of attraction $\rho(A)$, has full measure; and
\item there is no strictly smaller closed set $A'\subset A$ for which $\rho(A')$ has positive measure.  
\end{itemize}
Similarly, $A$ will be called a  \defin{topological attractor} if 
\begin{itemize}
\item its topological basin of attraction $\rho(A)$ is generic; and 
\item there is no strictly smaller closed set $A'\subset A$ so that $\rho(A')$ is not meager.  
\end{itemize} 
\end{definition}

%Note that if $A$ is a minimal attractor, then one actually has
%$$
%\omega(x)=A \qquad \text{ for almost all (or a not meager set) }x \in \rho(A).  
%$$

%\begin{remark}Although  attractors in the sense of the above definition describe the limiting dynamics in a topological sense, they are sometimes called \emph{metric} attractors in the literature, simply because the notion of measure is involve.  We prefer not to use this terminology, to avoid confusion with the following definition. 
%\end{remark}

The second possibility is to consider that attractors describe the asymptotic dynamics in a statistical way.  For a given point $x\in X$, consider the sequence of measures 
$$
\nu_{n}(x) = \frac{1}{n} \sum_{i<n} \delta_{T^{i}(x)}
$$ 
and call $\nu_{x}$ its limit in the weak* topology of measures, provided of course that this limit exists.  The measure $\nu_{x}$ thus describes the asymptotic statistical distribution of the system when started at $x$.  

\begin{definition}{\label{attractor_stat}}  A closed set $A\subset X$ is called a \defin{statistical (or physical) attractor}, if there exists a unique invariant measure $\mu$ such that: 
\begin{itemize}
\item[i)] $\text{supp}(\mu) = A$; and
\item[ii)]  $\mu = \nu_{x}$ for almost every $x$ (with respect to the reference measure $\lambda$). 
\end{itemize}
\end{definition}

%\todo{La notion de mesure de Lebesgue n'est pas si naturelle dans $\{0,1\}^\N$, est ce qu'on ne peut pas donner un cadre plus g\'en\'oral? }

The measure $\mu$ in the definition above describes the statistics of the limiting behavior when the system is started at a random initial condition. Such a measure is usually referred to as  \emph{Physical} \cite{SRB02}.

In addition to the  requirements of the previous definitions (in its topological or statistical versions), we are interested in how the computational complexity of the is affected by the dynamical properties of the system. Therefore, we will consider the following additional properties that constrain the behavior of the system in different ways:

\begin{definition}An attractor $A$ (of any kind) will be said to be 
\begin{itemize}
\item \defin{strongly attracting} if there exists an open neighbourhood $U$ of $A$ such that $T(\overline{U})\subset U$ and $A= \bigcap_{n} T^{n}(U)$; 
\item \defin{minimal} if the orbit of every point  in $A$ is dense in $A$;
\item \defin{transitive} if there exists a point in $A$ whose orbit is dense in $A$; 
\item \defin{exact} if every ball $B$ intersecting $A$ has an iterate  $ T^{n}B$ that covers $A$. 
\end{itemize}
\end{definition}

These are classical dynamical properties in the theory of dynamical systems. For example, attractors of hyperbolic systems are typically strongly attracting and transitive \cite{Bow75a}, whereas attractors in the logistic family are typically transitive and exact, and sometimes even minimal, but never strongly attracting unless they are finite. 

%
% \subsection{On the concept of limit sets}
% 
% There is different type of limit set, we propose here the different definitions.
% 
%\begin{definition}
% Let $T:X\to X$ be a continuous transformation over a compact set $X$ on which a natural reference measure $\Leb$.
% \begin{itemize}
%\item the \defin{limit set} is $\Omega(F)=\bigcap_n T^n(X)$;
%\item the \defin{generic limit set} $\Omega'(F)$ is the intersection of all closed subsets of $X$ with a comeager topological bassin;
%\item the \defin{$\Leb$-likely limit set} $\Omega_\Leb(F)$ is the intersection of all closed subsets of $X$ with a  topological bassin of measure one;
%\item the \defin{$\Leb$ limit set} $\Lambda\Leb(F)$ is the support of adherence values of $T^n\Leb$.
%\end{itemize}
%
% 
%\end{definition}
%

 \subsection{Rudiments of Computable Analysis}\label{section.ComputableAnalysis}

In this section, we recall notions from
computable analysis on metric spaces.
All the results presented here are well known.
For a detailed modern exposition we refer to \cite{Handbook}. In the following, we will make use of the word \textit{algorithm} to mean a computer program written in any standard programming language or, more formally, a Turing Machine \cite{Turing1936}. Algorithms are assumed to be only capable of manipulating integers. By identifying countable sets  with integers in a constructive way, we can let algorithms work on these countable sets as well.  For example, algorithms can manipulate rational numbers by identifying each $p/q$ with some integer $n$ in such a way that both $p$ and $q$ can be computed from $n$, and vice-versa. We fix such a numbering from now on.

\subsubsection{Computable metric spaces.}

\begin{definition}
A \defin{computable metric space} is a triple $(X,d,\mathcal{S})$, where $(X,d)$ is a metric space
and $\mathcal{S} = \{s_i : i \ge 0\}$ a countable dense subset of $X$,
whose elements are called \defin{ideal points},
such that there exists an algorithm
which, upon input $(i,j,n) \in \N^{3}$,
outputs  $r\in \Q$ such that
$$|d(s_i,s_j)-r| \le 2^{-n}.$$
We say that the distances between ideal points
are \emph{uniformly computable}.
\end{definition}

For $r>0$ a
rational number and $x$ an element of $X$,
we denote by $B(x,r) = \{ z \in X \ : \ d(z,x)<r\}$ and by $\overline{B}(x,r) = \{ z \in X \ : \ d(z,x)\leq r\}$
respectively the open and closed ball with center $x$ and radius $r$. The balls centered on elements of
$\mathcal{S}$ with rational radii are called \defin{ideal balls}. A computable enumeration of the ideal balls $B_{n}=B(s^{(n)},r^{(n)})$ can be obtained by taking for instance a bi-computable bijection $\varphi:\N \rightarrow
\N \times \Q$ and letting $s^{(n)} = s_{\varphi_1 (n)}$ and $r^{(n)} = \varphi_2 (n)$, where $\varphi(n) = (\varphi_1 (n),\varphi_2 (n))$.  We fix such a computable enumeration from now on. For any
subset $I$ (finite or infinite)
of $\N$, we denote $\mathcal{U}_{I}$
the collection of ideal balls $B_n$ with $n \in I$, and
$U_{I}$ the union of these balls:
\[U_{I} = \bigcup_{n \in I} B_n.\]
An open set $V\subset X$ is \defin{lower computable (or r.e.)} if $V=U_I$ for some recursively enumerable (r.e.) $I\subset \N$. That is, there is an algorithm which halts on some $n\in\N$ if and only if $n\in I$. A point $x\in X$ is \defin{computable} if $\{x\}=\bigcap_{n\in I} B_n$ for  some r.e. $I\subset\N$.

\begin{example}The following examples will be important for us. In particular, example $c)$ provides the definition of computable probability measure.  
\begin{itemize}
\item[a)] For a finite alphabet  $\mathcal{A}$
and $0$ one of its elements, the Cantor space $\mathcal{A}^{\N}$ with its usual
metric has a natural computable metric space structure where the ideal points can be taken to be $\mathcal{S} = \{w0^{\infty}:
w \in \mathcal{A}^{*}\}$, where $\mathcal{A}^*$ denotes the set of all finite words one can make with the symbols of $\mathcal{A}$.  In this case, the ideal balls are
the cylinders (see \ref{symbolic_notation}).
\item[b)] The compact interval $[0,1]$, with its usual metric
and $\mathcal{S}=\Q\cap [0,1]$ is also
a computable metric space. The ideal balls here are the
open intervals with rational endpoints.
\item[c)] If $X$ is a compact computable metric space, then the set $M_X$ of probability measures over $X$ can also be made a computable metric space \cite{GalHoyRoj07c} with a metric that induces the weak topology on $M_X$. In particular, we automatically obtain a notion of \defin{computable probability measure}, namely, those corresponding to the computable elements of this space.  
\end{itemize}
\end{example}

Regarding the notion of computable measure alluded to in the preceding example, we will only use the following characterization (see for example \cite{HoyRoj07}). 

\begin{proposition}\label{prop:measures}Let $\mu$ be a probability measure over a computable metric space $X$. Then, $\mu$ is computable if and only if $\mu(B_{i_1}\cup B_{i_2}\cup \dots \cup B_{i_n})$ is lower computable uniformly in $i_1,\dots, i_n$ in the sense that there exists an algorithm which, upon input  $i_1,\dots, i_n$ and $k$, outputs a rational number whose supremum over $k$ equals $\mu(B_{i_1}\cup B_{i_2}\cup \dots \cup B_{i_n})$.
\end{proposition}

Let $(X,d)$ and $(X',d')$ be computable metric spaces. Let us denote by $(B'_m)_m$
an enumeration of ideal balls of $X'$. A function $T : X \rightarrow X'$ is
\defin{computable} if there exists an algorithm
which, given as input some integer $m$,
enumerates a set $I_m$
such that $$T^{-1} (B'_m) = U_{I_m}.$$

It follows that computable functions are continuous. It is perhaps more intuitively familiar, and provably equivalent, to think of a computable function as one for which there is an algorithm which, provided with arbitrarily good approximations of $x$, outputs arbitrarily good approximations of $T(x)$.  In symbolic spaces, this can easily be made precise:

\begin{example}
In Cantor space $\mathcal{A}^{\N}$, a function
$T : \mathcal{A}^{\N} \rightarrow \mathcal{A}^{\N}$ is computable if and only if there exists some non decreasing
computable function $\varphi : \N \rightarrow \N$ together with an algorithm which, provided with the $\varphi(n)$ first symbols of the sequence $x$, computes the $n$th first symbols of the sequence $T(x)$.
\end{example}

A function $T$ is \defin{effectively open} if there is an algorithm which given as input some integer $m$, enumerates a set $I_m$ such that $T(B_m)=U_{I_m}$.

A compact set $K\subset X$ is said to be \defin{recursively compact}
if the inclusion $$K \subset U_I,$$
where $I$ is some finite subset of $\N$, is semi-decidable. That is, if there is an algorithm which, given  $I$ as input, halts if and only if
the inclusion above is verified. 

\begin{example}
The Cantor space and the compact interval
are easily seen to be recursively compact.
\end{example}

%\subsubsection{Computable closed subsets}

A closed subset $E \subset X$ is said to be \defin{upper computable (or effective, or co-r.e.)} if its complement $X \backslash K $ is a lower computable open set.  It is called \defin{lower computable} if the relation $E\cap B_i\neq \emptyset$ can be semi-decided uniformly for all ideal balls $B_i$. Finally, $E$ is called \defin{computable} if it is both lower and upper computable. 

We will also make use of the following fact. We include a proof for the reader's convenience.  

\begin{proposition}
\label{prop.equivalence.comp.closed.subset}
Let $(X,d,\mathcal{S})$ be a recursively compact
computable metric space.
A closed subset $K \subset X$ is  upper computable  if and
only if it is recursively compact. 
%the inclusion
%$$K \subset U_I \qquad \text{ where } I \subset \N \text{ is finite }$$
%is semi-decidable.
\end{proposition}
\begin{proof} Suppose $K$ is recursively compact and let $\bar{B}\subset X$ be a closed ideal ball. Note that since $X\setminus \bar{B}$ is effectively open, by the recursive compactness of $K$ we can semi-decide whether $K\subset (X\setminus \bar{B}) \iff K\cap \bar{B} = \emptyset$, and thus recursively enumerate all ideal balls $B$ with this property. Noting that their union equals $X\setminus K$, we see that $K$ is upper computable. Now, let $U\subset X$ be an effectively open set. We show that we can semi-decide whether $U$ covers $K$. Note that $U$ covers $K$ exactly when $U\cup (X\setminus K)$ covers $X$. Thus, if $K$ is upper computable, then $U\cup (X\setminus K)$ is effectively open, and the desired result follows from  recursive compactness of $X$.   
\end{proof}

\begin{example} If $X$ is a recursively compact computable metric space, then the set $\mathcal{M}_X$ of probability measures is also a recursively compact computable metric space.  Moreover, if $T:X\to X$ is a computable map, then the set of invariant measures $\mathcal{M}_T$ is upper computable, and therefore recursively compact as well (see \cite{HoyRoj09}).   
\end{example}

\subsubsection{Computational structure of closed sets}

Previous results on the computability of attractors imply that they are in general non-computable \cite{GZB12, BY06}.  Thus, since we are interested in describing the general computational structure that attractors possess, we need notions able to capture their levels of non-computability.  We will do this by carefully identifying the information that is actually hard to compute, and measuring its computational hardness through the complexities of their integer representations in the usual \textit{arithmetical hierarchy} see~\cite{Rogers} for standard definitions. Let us briefly recall the definition here. A set $A\subset\N$ is said to be \textit{$\Pi_n$-computable} if there exists a computable map $f:\N^{n+1}\to\{0,1\}$ which satisfies
\[i\in A\Longleftrightarrow \underset{\text{$n$ alternating quantifiers}}
{\underbrace{\forall i_1,\exists i_2,\forall i_3,\dots}}
f\left(i,i_1,\dots,i_k\right)=1,\]
and  \textit{$\Sigma_n$-computable} if there exists a computable map $f:\N^{n+1}\to\{0,1\}$ which satisfies
\[i\in A\Longleftrightarrow \underset{\text{$n$ alternating quantifiers}}
{\underbrace{\exists i_1,\forall i_2,\exists i_3,\dots}}
f\left(i,i_1,\dots,i_k\right)=1.\]

The definition of computability for compact sets, say in the plane for simplicity, is equivalent to having an algorithm that can draw the set on a computer screen with arbitrary precision. Roughly, this means that one can zoom-in at any point in the set to produce arbitrarily high magnifications of it. In order to achieve this, one must be able to determine, at any given resolution, whether a pixel must be coloured or not. Lower computability represents the task of detecting whether a given pixel must be coloured, and upper computability the task of detecting whether it has to be left uncoloured. The following definition is intended to measure the \textit{degree of unsolvability} of either of these tasks for a given compact set.

Let $K$ be a compact subset of a computable metric space $X$. The \defin{inner (or intersecting)} and \defin{outer}  collections of ideal balls of $K$, denoted respectively by $\mathcal{B}(K)$ and $\mathcal{\overline{B}}(K)$, are defined by  
$$
\mathcal{B}_{in}(K) = \left\{ i\in\N :  B_{i}\cap K \neq \emptyset \right\} \qquad \text{ and }\qquad \mathcal{\overline{B}}_{out}(K) = \left\{ i\in\N :  \overline{B_{i}}\cap K = \emptyset \right\}.
$$ 
We remark that in general the closure of the ball $B_i$ may be different from the closed ball $\overline{B}_i$. 

 We will measure the computational cost of describing closed  subsets of $X$ by measuring the complexity of computing their inner and/or outer collections of ideal balls.  

\begin{definition}
Then $K$ is said to be $\defin{\Sigma_{n}}$ (respectively $\defin{\Pi_{n}}$) if its inner (respectively outer) collection is $\Sigma_{n}$.  If $K$ is both $\Sigma_{n}$ and $\Pi_{n}$, we simply say that it is $\Delta_{n}$. 
\end{definition}

%Note that a $\Sigma_{1}$ (respectively $\Pi_{1}$) closet set is just a lower (respectively upper) semi-computable closed set. In fact, it is not hard to see that a closed set is $\Sigma_{n}$ (respectively $\Pi_{n}$) if it is lower (respectively upper) semi-computable relative to the $n-1$-th jump of the halting problem. For this reason,  we will also refer to a $\Sigma_{n}$ (or $\Pi_{n}$) closed set as being \defin{lower (or upper) semi-computable of order $\defin{n}$}. 

We will say that $K$ is \defin{$\Sigma_{n}$-complete} if it is $\Sigma_{n}$ and any other $K'$ can be \defin{reduced} to $K$ in the sense that, provided with an enumeration of the inner collection of $K$, we can enumerate the inner collection of  $K'$. A  $\Pi_{n}$-\defin{complete} closed set is defined in a similar way. 

The following proposition shows that the hierarchy of closed sets one obtains with the previous definition is consistent with the arithmetic hierarchy.  

\begin{proposition}
If $K$ is $\Sigma_{n}$ (respectively $\Pi_{n}$), then it is $\Pi_{n+1}$ (respectively $\Sigma_{n+1}$).
\end{proposition}

\section{Computability obstructions in attractors realization: complexity upper bounds}\label{section.UpperBounds}

In this section we provide general obstructions that bound the maximal complexity attractors may exhibit according to their types. 

%\subsection{Upper complexity bounds of attractors}\label{section.UpperBounds}
%
%\begin{proposition}
%Let $A$ be a metric or a topological attractor. Then, the set 
%$$
%\rho'(A)=\{x: \omega(x)=A\}
%$$
%has full measure or is generic.  
%\end{proposition}
%\begin{proof}
%Suppose $A$ is a topological attractor. Then, its realm of attraction contains a dense $G_\delta$-set, call it $G$. Now, let $B$ be any open ball. Note that $B\cap G$ cannot be a meager set, and thus, by the second requirement of Definition \ref{attractor_top}, the set $\omega(B\cap G)$ of accumulation points of orbits started in $B\cap G$ must be equal to $A$. 
%\end{proof}

\begin{theorem}\label{upper_bounds} Let $X$ be a recursively compact computable metric space, $T:X\longrightarrow X$ be a computable map, $\Leb$ a computable reference measure and $A$ be a closed subset of $X$ which is invariant by $T$. 
\begin{itemize}
\item[(i)] If $A$ is a topological attractor, then it is a $\Pi_{2}$ set. 
\item[(ii)] If $A$ is a metric attractor, then it is a $\Pi_{2}$ set. 
\item[(iii)]  If $A$ is a statistical attractor, then it is $\Sigma_{2}$. 
\item[(iv)] If $A$ is strongly attracting, then it is $\Pi_{1}$.
\end{itemize}

Now, assume that $A$ is $\Pi_{n}$ with $n=1, 2$.  

\begin{itemize}
\item[(v)] If $A$ is minimal, then it is also $\Sigma_{n}$.
\item[(vi)]  Suppose in addition that $T$ is effectively open. If the action of $T$ on $A$ is exact, then $A$ is also $\Sigma_{n}$. 
\end{itemize}

\end{theorem}
\proof

 For a given ideal ball $B_{i}$, denote by $\tilde{\rho}(\overline{B_{i}})$ the set of points $x$ such that $\omega(x)$ intersects $\overline{B_{i}}$.  That is, 
$$
\tilde{\rho}(\overline{B_{i}}) =\left\{x\in X: \omega(x)\bigcap\overline{B_{i}}\ne\emptyset\right\}=\bigcap_k\bigcap_N\bigcup_{n\geq N} T^{-n}B\left(\overline{B_{i}},\frac{1}{k}\right)$$
where $B(K,\epsilon)=\{x\in X:d(x,K)<\epsilon\}$ is an open upper approximation of $K$ at precision $\epsilon$. Note that $\tilde{\rho}(\overline{B_{i}})$ is a $G_{\delta}$ set (not necessarily dense or of positive measure).  We remark that $\rho(A)\subset\tilde{\rho}(A)$.

\medskip
\noindent \emph{Proof of $(i)$}: Let $A$ be the topological attractor of the map $T$. We first show that the closure of an ideal ball $\overline{B_i}$ does not intersect $A$ if and only if $\tilde{\rho}(\overline{B_i})$ is not dense. 

If $\overline{B_i}\cap A=\emptyset$, the orbits started at points in $\rho(A)$ must accumulate in $A$ and so can visit $B(\overline{B_i},\epsilon)$ only finitely many times for sufficiently small $\epsilon$. In other words $\tilde{\rho}(\overline{B_i})\cap\rho(A)=\emptyset$.  By definition, the realm of attraction of $A$, $\rho(A)$, contains a dense $G_{\delta}$ set, thus $\tilde{\rho}(\overline{B_i})$ cannot be dense by Baire's category theorem.

Conversely, if $\tilde{\rho}(\overline{B_i})$ is not dense, there exists a ball $B$ which does not intersect $\tilde{\rho}(\overline{B_i})$. But $B\cap\rho(A)$ is not meager so by minimality of $A$, which is the second requirement of Definition \ref{attractor_top}, one has $\omega(B\cap\rho(A))=A$. It follows that $\overline{B_i}\cap A=\emptyset$ since an element of $B$ visits $B(\overline{B_i},\epsilon)$ only finitely many times for  sufficiently small $\epsilon$. 

Now, for a given ideal ball $B_i$, one has
\begin{eqnarray*}
 \overline{B}_i\cap A\ne\emptyset&\Longleftrightarrow&  \tilde{\rho}(\overline{B_i})\textrm{ dense }\\
&\Longleftrightarrow& \forall j\in\N,\quad  \tilde{\rho}(\overline{B_i})\cap B_j\ne\emptyset\\
&\Longleftrightarrow& \forall j,k, n\in\N,\exists m\in\N, \quad   \bigcup_{t=n}^m T^{-t}B\left(\overline{B_{i}},\frac{1}{k}\right)\cap B_j\ne\emptyset
\end{eqnarray*}

It is clear that $B\left(\overline{B_{i}},\frac{1}{k}\right)$ is a lower computable open set. Moreover, since $T$ is computable, we can enumerate,  uniformly in $i,k,n$ and $m$, ideal balls whose union is $\bigcup_{t=n}^m T^{-t}B\left(\overline{B_{i}},\frac{1}{k}\right)$. As  the intersection of lower computable open sets is again lower computable, we deduce that it is uniformly semi-decidable whether the instersection is non empty. It follows that $\{i\in\N:\overline{B_i}\cap A\ne\emptyset\}$ is a $\Pi_2$-computable set.

%Let $A$ be the topological attractor of the map $T$. This means that, in particular, the complement of $A$ consists of the set of wandering points, where a point is said to be wandering it has a neighbourhood $U$ such that  
%$$
%\exists\, N\in\N:\, \forall\, n\geq N,\, T^{n}U\cap U = \emptyset.  
%$$
%In oder words, $x$ is in the complement of $A$ if and only if 
%$$
% \exists\, \epsilon \in \Q,\, \exists\, N \in \N:\, \forall\, n\geq N,\, T^{n}B(x,\epsilon)\cap B(x,\epsilon) = \emptyset. 
%$$
%It follows that 
%$$ X\setminus A = \displaystyle{\bigcup_{\{B : \exists\, N \in \N:\, \forall\, n\geq N,\, T^{n}B \cap B = \emptyset \}}} B .$$   
%
%Now, given $T^{n}B\cap B = \emptyset$, it is easy to see that there is a possibly smaller ball $B'$ such that the stronger statement $T^{n}\overline{B}\cap B = \emptyset$ holds. Since $T$ is computable and $\overline{B}$ is recursively compact, the images $T^{n}\overline{B}$ are recursively compact too, which implies that the relation $T^{n} \overline{B}  \cap B = \emptyset$ can be semi-decided. Therefore, we see that, with the help of the Halting problem, we can enumerate all such balls. The claim follows.  
\medskip
\noindent \emph{Proof of $(ii)$}:  Let $A$ be the metric attractor of the map $T$. We first show that an ideal ball $\overline{B_i}$ does not intersect $A$ if and only if $\tilde{\rho}(\overline{B_i})$ has null measure.

If $\overline{B_i}\cap A=\emptyset$, one has $\tilde{\rho}(\overline{B_i})\cap\rho(A)=\emptyset$.  By definition, the realm of attraction of $A$ has measure one so $\lambda(\tilde{\rho}(\overline{B_i}))=0$.

Conversely, assume that $\overline{B_i}\cap A\ne\emptyset$. Since $A$ is the only metric attractor, by the second requirement of Definition~\ref{attractor_top}, we deduce that for all ideal balls $B_j$ such that $A\cap B_j^c\ne\emptyset$, where $B_j^c=X\setminus B_j$, one has $\lambda(\rho(A\cap B_j^c))=0$. Thus, 
$$
\lambda\left(\bigcup_{\{j \, : \,  B_j\cap A \neq \emptyset\}} \rho(A\cap B_j^c)\right) = 0. 
$$
But this is exactly the set of points $x\in X$ such that $\omega(x)\neq A$.  It follows that $\omega(x)=A$ for almost every $x$. Since $\overline{B_i}\cap A\ne\emptyset$, one deduces that the trajectory of $x$ visits $B\left(B_i,\frac{1}{k}\right)$ infinitely many times for arbitrary $k$ so $\lambda(\tilde{\rho}(B_i))>0$.

Now, for $B_i$ an ideal ball, one has
 \begin{eqnarray*}
 \overline{B_i}\cap A\ne\emptyset&\Longleftrightarrow&  \overline{B_i}\cap A\ne\emptyset \\
&\Longleftrightarrow&  \lambda(\tilde{\rho}(\overline{B_i}))>0\\
&\Longleftrightarrow& \forall j,k, n\in\N, \exists m\in\N, \quad   \lambda\left(\bigcup_{t=n}^mT^{-t}\left(B\left(\overline{B_i},\frac{1}{k}\right)\right)\cap B_j\right)>0.
\end{eqnarray*}

Since $T$ is computable, we can uniformly enumerate ideal balls whose union is $\bigcup_{t=n}^mT^{-t}\left(B\left(\overline{B_i},\frac{1}{k}\right)\right)\cap B_j$. Thus to semi-decide if this set has positive measure, it is sufficient to find one of these balls with positive measure and this is semi-decidable by Proposition \ref{prop:measures}. It follows that $\{i\in\N:\overline{B_i}\cap A\ne\emptyset\}$ is a $\Pi_2$-computable set.

\medskip
\noindent \emph{Proof of $(iii)$}: Let $A$ be the statistical attractor of the map $T$ and $\mu$ be the invariant measure supported on $A$, that satisfies  $\mu = \nu_{x}$ for almost every $x$ w.r.t $\lambda$. Given a continuous function $\varphi:X\to\R$, one has

\begin{eqnarray*}
\int \varphi(x)d\mu(x)&=&\int\left(\int \varphi(y)d\nu_x(y)\right)d\lambda(x)\\
&=&\int\lim_{n\to\infty}\frac{1}{n}\sum_{t=0}^{n-1}\varphi(T^t(x))d\lambda(x)\\
&=&\lim_{n\to\infty}\int \varphi(x)d\lambda_n(x)
\end{eqnarray*}
where $\lambda_n=\frac{1}{n}\sum_{t=0}^{n-1}T^t_{*}\lambda$. Since $T$ is computable, $\mu$ is the weak limit as $n\to\infty$ of the uniformly computable sequence $(\lambda_n)_{n\in\N}$. By Proposition \ref{prop:measures} and the fact that the complement of a closed ideal ball is effectively open, it follows that there exists an algorithm which takes as input $n,k$ and $l$ and returns a rational number $q_{n,k,l}$ whose infimum over $l$ equals $\lambda_n(\overline{B_k})$. Now, for $B_i$ an ideal ball, one has
 \begin{eqnarray*}
 B_i\cap A\ne\emptyset&\Longleftrightarrow& \mu(B_i)>0\\
 &\Longleftrightarrow& \exists j,k,m\in\N, \forall n\geq m,\quad \overline{B_k}\subset B_j\textrm{ and }\lambda_n(\overline{B_k})\geq 2^{-k}\\
  &\Longleftrightarrow& \exists j,k,m\in\N, \forall n\geq m,\forall l\in\N,\quad \overline{B_k}\subset B_j\textrm{ and }q_{n,k,l}\geq 2^{-k}.
 \end{eqnarray*}

 We deduce that $\{i\in\N:B_i\cap A\ne\emptyset\}$ is a $\Sigma_2$-computable set. 

\medskip

\noindent \emph{Proof of $(iv)$} Assume $A$ is strongly attracting with neighborhood $U$. Since $A$ is compact, there must exists a finite cover of $A$ by ideal balls $B_{n_{1}},\dots,  B_{n_{k}}$ such that 
$$
\mathcal{B}=\bigcup_{i=1}^{k}\overline{B_{n_{i}}} \subset U. 
$$
Since $\mathcal{B}$ is a computable closed set, it is in particular recursively compact. We see then that $A=\bigcap_{n}T^{n}\mathcal{B}$ is also recursively compact, which by Proposition \ref{prop.equivalence.comp.closed.subset} is therefore upper semi-computable, as it was to be shown.

\medskip
\noindent \emph{Proof of $(v)$} Assume $A\subset X$ is a closed invariant set on which $T$ is minimal. Let $\tau$ be a computable surjection from $\N$ to the finite subsets     of $\N$. Consider an ideal ball $B_i$, from the minimality property it follows that 
 \begin{eqnarray*}
 B_i\cap A \neq \emptyset &\Longleftrightarrow&  A \subset \bigcup_{t\geq 0 } T^{-t}(B_i)\\
 &\Longleftrightarrow& \exists k\in\N, \forall j\in\tau(k),\quad \overline{B_j}\cap A=\emptyset \\
 &&\textrm{ and }\left(\bigcup_{j\in\tau(k)}B_j\right)\cup\left(\bigcup_{t\geq 0 } T^{-t}(B_i)\right)=X.
\end{eqnarray*}
 Assume that $A$ is $\Pi_n$-computable with $n=1$ or $2.$ We deduce that it is $\Sigma_n$-computable to decide  that $\overline{B_j}\cap A=\emptyset $. Moreover, since $T$ is computable, it is possible to enumerate ideal balls whose union is $\bigcup_{t\geq 0 } T^{-t}(B_i)$. Using the fact that $X$ is recursively compact, it is uniformly semi-decidable to know if 
\[\left(\bigcup_{j\in\tau(k)}B_j\right)\cup\left(\bigcup_{t\geq 0 } T^{-t}(B_i)\right)=X.\]
 
 We deduce that $\{i\in\N:B_i\cap A\ne\emptyset\}$ is a $\Sigma_n$-computable set.

\medskip
\noindent \emph{Proof of $(vi)$} Finally, assume $T$ is open and exact on $A$. It follows that for an ideal ball $B_i$,  
$$
B_i\cap A \neq \emptyset \iff A \subset \bigcup_{t}T^{t} (B_i). 
$$ 
Since $T$ is effectivelly open, by hypothesis, the set $\bigcup_{t}T^{t}(B_i)$ is recursively open uniformly in $i$. The claim now follows exactly as in the previous case.

\endproof

\section{Realizing complex attractors within symbolic systems}\label{section.symbolic}

In this section we present  examples of computable systems on symbolic spaces exhibiting attractors whose algorithmic complexity match the general upper bounds given by Theorem \ref{upper_bounds}, showing that these bounds are tight in general.

\subsection{Symbolic Systems}\label{symbolic_notation}

Given a finite alphabet $\mathcal{A}$, a \defin{word} or \defin{pattern} $x_1x_2\dots x_n$ is a finite sequence (possibly empty) of symbols from $\A$. We denote by $\A^{\ast}$ the set of words and by $\A^{+}$ the set of non empty words. The \defin{length} of a word $u\in\A^{\ast}$ will be denoted by $|u|$.  A symbolic space is a set of the form $\mathcal{A}^\N$, whose elements will be referred to as \defin{configurations}. Given $i\leq j$, the restriction of configuration $x\in\mathcal{A}^\N$ to $[i,j]$ is defined to be the word $x_i\dots x_j$, and denoted $x_{[i,j]}$. We say that a word $u\in\A^{\ast}$ \defin{appears} at position $i\in\N$ in a configuration $x\in\mathcal{A}^\N$ if $x_{[i,i+|u|-1]}=u$. Given $A\subset\mathcal{A}^\N$, the \defin{language} of $A$ is the set of words that appear in some configuration of $A$. The \defin{cylinder} centered on the word $u\in\A^{\ast}$ at the position $i\in\N$ is the set of configurations defined by $$[u]_i=\{x\in\mathcal{A}^\N: u \text{ appears at position } i \text{ in }x \}.$$

It is well known that $\mathcal{A}^\N$ is compact for the product topology (where $\mathcal{A}$ is considered with the discrete topology) and that the cylinders form a basis of clopen sets. This is a computable space: the ideal points can be taken to be the set of all eventually constant configurations, and the distance is given by
$$d(x,y)=2^{-\min\{n\in\N:x_n\ne y_n\}}\textrm{ for all }x,y\in\A^\N.$$ 
  
According to the definition of computable maps, $T:\mathcal{A}^\N\longrightarrow\mathcal{A}^\N$ is computable iff there exist two computable functions $\alpha:\N\longrightarrow\N$ and $\beta:\mathcal{A}^{\ast}\longrightarrow\mathcal{A}$ such that $$T(x)_n=\beta(x_{[0,\alpha(n)]})\textrm{ for all $x\in\mathcal{A}^\N$ and $n\in\N$}.$$

A simple example of a computable map on $\mathcal{A}^\N$ is the \defin{shift} map defined by
$$\begin{array}{ccccc}
\sigma:&\A^{\N} & \longrightarrow &\A^{\N}&\\
&x&\longmapsto & \sigma(x)& \textrm{ where } \sigma(x)_i=x_{i+1}\textrm{ for all }i\in\N.
\end{array}$$

A closed shift-invariant subset of $\A^\N$ is called a \defin{subshift}. Equivalently, a subshift can be defined by a list of forbidden patterns as the collection of all configurations in which no forbidden pattern appears. If a subshift $A$ is defined by forbidden patterns that can be enumerated by a Turing machine, then $A$ is $\Pi_1$. If the forbidden patterns can be enumerated with the help of the Halting problem, then $A$ is $\Pi_2$.

Given a symbol $a\in \{0,1\}$, a \defin{block} of size $n$ for $a$ is a finite word of the form $$b\underbrace{aaaa\dots a}_{n}c$$ where $b$ and $c$ are different from $a$. The position of the block in a configuration is the index of the letter $b$.

\subsection{A $\Pi_1$-complete Strong Attractor}

We present here an example of a computable system having a $\Pi_1$-complete invariant set $A$, which is the unique attractor of the system in the strongest possible sense: topological, metric and statistical. Moreover, it is strongly attracting, and the dynamics it supports is highly non trivial. Note that by Theorem \ref{upper_bounds}, such attractors must be at most $\Pi_1$.

%More precisely, we show the following theorem. 

\begin{theorem}\label{thm:Pi_1-complete}
There exists a computable map $T: \{0,1\}^\N\to \{0,1\}^\N $ and an invariant closed subset $A\subset X$ with the following properties: 
\begin{itemize}
\item[i)] $A$ is $\Pi_1$-complete.
\item[ii)] $A$ is the topological and metric attractor of $T$; Moreover,  the dynamics of $T$ on $A$ is a topologically mixing shift of positive entropy. 
\item[iii)] $A$ is strongly attracting, and its realm of attraction is in fact the whole space: $\omega(x)\subset A$ for all $x\in X$;
\item[iv)] $A$ is the statistical attractor of $X$ w.r.t any reference measure $\lambda$ which is shift-ergodic and of full support. In particular, the unique physical measure of the system is non computable.  
\end{itemize}
\end{theorem}

\begin{proof} 
 
Consider the set  $$\mathcal{H}alt = \{i \in \N : M_{i} \text{ halts on the empty input}\}$$ of indices of machines that halt on the empty input. It is well known that its complement $\overline{\mathcal{H}alt}$ is a $\Pi_1$-complete set.

\medskip 

\noindent\textbf{Definition of the system.} Let $X=\{0,1\}^{\N}$.  Consider the subshift $A\subset X$ defined by the forbidden pattens of the form $01^{n}0$  where $n\in\mathcal{H}alt$. More precisely, any element of $A$ has the following form
$$
 1^{*}0^{k_{1}}1^{n_{1}}0^{k_{2}}1^{n_{2}}\dots 0^{k_{i}}1^{n_{i}} \dots  
$$
where the $k_{i}$'s are arbitrary, and the $n_{i}$'s are such that machine $M_{n_{i}}$ does not halt on the empty input. 

We construct a computable function $T$ over $X$ for which $A$ is a strong attractor on which $T$ acts as the shift map. The idea to achieve this is simple, we want successive iterations of $T$ to have, in the limit,  the following effect on a given input $x$: to read the blocks of $1$'s, to erase those whose length is the index of a Turing machine that halts, and then make a shift.  Now for the details.

The definition of $T$ is as follows. Let $(M_{e})_{e\in\N}$ be an enumeration of all Turing Machines. Let $x\in X$. Then the $i^\textrm{th}$ bit of $T(x)$, denoted by $T(x)_{i}$, is given by:

\begin{itemize}
	\item  if there exist $j_1,j_2$ such that \begin{enumerate}\item[i)] $j_{1}\leq i < j_{2} \leq 2i$, \item[ii)] $x_{[j_{1},j_{2}]}=01^{l}0$ with $l\leq j_{1}$  and \item[iii)] $M_{l}$ halts on the empty input in at most $j_{1}$ steps,\end{enumerate} then $T(x)_{i}=0$;
		
	\item else $T(x)_{i}=x_{i+1}$. 
\end{itemize}

In other words, first the function $T$ erases in $x$ any block of 1's of length $l$ if it is at a position $i\geq l$ and the machine $M_{l}$ halts on the empty input in at most $i$ steps, and then it shifts the configuration obtained. In particular any block which is not deleted after the first iteration of $T$ is shifted, and will keep being shifted in subsequent iterations until it disappears (see Figure \ref{fig:Pi_1} for an illustration).

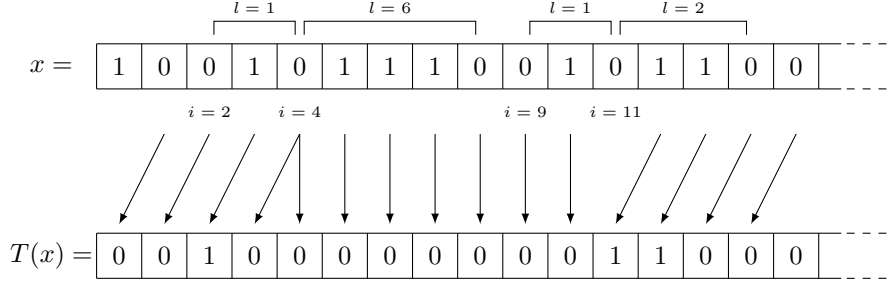
\begin{figure}
    \centering
\begin{tikzpicture}[scale=0.6]

 \draw(-1,0.5) node{$x=$};

\draw (0,0)--(16.5,0);
\draw (0,1)--(16.5,1);

\draw[dashed] (17.5,0)--(16.5,0);
\draw[dashed] (17.5,1)--(16.5,1);

\foreach \i in{0,1,...,16}{
\draw(\i,0)--+(0,1);}
 
\foreach \i in{1,2,4,8,9,11,14,15}{
\draw(\i+0.5,0.5) node{$0$};}
\foreach \i in{0,3,5,6,7,10,12,13}{
\draw(\i+0.5,0.5) node{$1$};}

\begin{tiny}
 \draw(2.5,-0.5) node{$i=2$};
 \draw(4.5,-0.5) node{$i=4$};
 \draw(9.5,-0.5) node{$i=9$};
 \draw(11.5,-0.5) node{$i=11$};

 \draw(3.5,1.8) node{$l=1$};
 \draw(6.5,1.8) node{$l=6$};
 \draw(10.5,1.8) node{$l=1$};
 \draw(13,1.8) node{$l=2$};
\end{tiny}

\draw (2.6,1.2)--(2.6,1.5)--(4.4,1.5)--(4.4,1.2);

\draw (4.6,1.2)--(4.6,1.5)--(8.4,1.5)--(8.4,1.2);

\draw (9.6,1.2)--(9.6,1.5)--(11.4,1.5)--(11.4,1.2);

\draw (11.6,1.2)--(11.6,1.5)--(14.4,1.5)--(14.4,1.2);

\foreach \i in{1,2,3,4,12,13,14,15}{
\draw[->,>=latex](\i+0.5,-1)--++(-1,-2);}
\foreach \i in{4,5,...,10}{
\draw[->,>=latex](\i+0.5,-1)--++(0,-2);}

\begin{scope}[yshift=-4.2cm]
 \draw(-1,0.5) node{$T(x)=$};

 \draw (0,0)--(16.5,0);
\draw (0,1)--(16.5,1);

\draw[dashed] (17.5,0)--(16.5,0);
\draw[dashed] (17.5,1)--(16.5,1);

\foreach \i in{0,1,...,16}{
\draw(\i,0)--+(0,1);}
 
\foreach \i in{0,1,3,4,5,6,7,8,9,10,13,14,15}{
\draw(\i+0.5,0.5) node{$0$};}
\foreach \i in{2,11,12}{
\draw(\i+0.5,0.5) node{$1$};}
\end{scope}

\end{tikzpicture}
    \caption{An illustration of the action of $T$ on a configuration $x$. It is assumed that machines $M_1$ and $M_3$ halt in 8 and 4 steps respectively, while machine $M_2$ doesn't halt at all.}
    \label{fig:Pi_1}
\end{figure}

To see that $A$ is $\Pi_1$-complete set we simply note that
$$A\cap[01^n0]_0\ne\emptyset\Longleftrightarrow n\in \overline{\mathcal{H}alt}. $$

\noindent{\textbf{Properties of $A$.}}  It is clear that $A$ is a topologically mixing subshift of positive entropy and that $T$ acts on $A$ as the shift. We start by showing that $A$ is strongly attracting since in fact it attracts every orbit. 

\Claim{One has
 $$A=\bigcap_{n\in\N}T^n(X),$$
 in particular $\omega(x)\subset A$ for all $x\in X$.}
\bclaimprf
Since $A$ is $T$-invariant, one immediately has $A\subset\bigcap_{n\in\N}T^n(X).$
Let $x\in\bigcap_{n\in\N}T^n(X)$ and suppose $x\notin A$. Then, there exists $j\in\N$ such that $x_{[j,j+l+1]}=01^l0$ and the machine $M_{l}$ halts on the empty input, say,  in less than $t$ steps.  Pick $n\in\N$  such that $n+j\geq \max(t,l)$. Then there exists $y\in X$ for which $T^n(y)=x$.  Note that if  $T(z)$ has a block of $1$'s at position $i$ for a given $z$,  then $z$ must have a block of $1$'s at position $i+1$. Therefore, since $x_{[j,j+l+1]}=01^l0$, one must have $y_{[n+j,n+j+l+1]}=01^l0$. However $T(y)_{[n+j-1,n+j+l]}=0^{l+2}$ since $M_{l}$ halts on the empty input in less than $t$ steps and $n+j\geq \max(t,l)$. We deduce that $x_{[j,j+l+1]}=0^{l+2}$, which is a contradiction.
\eclaimprf

\medskip

We now show that $A$ is the topological and metric attractor of $T$. We do this by showing that for a full measure $G_{\delta}$-dense set $\Lambda$ of initial conditions, we have $\omega(x) = A$, which implies that there cannot be a proper subset of $A$ whose topological basin of attraction is not meager or of positive measure.  

Define $\Lambda\subset X$ as the set of configurations on which any word in the language of $A$ appears infinitely often. Clearly, $\Lambda$ is a $G_{\delta}$ dense set, and by the ergodic theorem, it has  full measure with respect to every shift-ergodic measure of full support. 

\Claim{For every $x\in \Lambda$ one has $\omega(x)=A$.}
 \bclaimprf
 Let $x$ be a configuration in $\Lambda$ and let $u$ be in the language of $A$, there exists $v$ in the language of $A$ which begin and end by $0$ and contains $u$ as subword at the position $k$. 
 By definition of $\Lambda$, there exists a strictly increasing sequence $(i_n)_{n\in\N}$ such that $x_{[i_n,i_n+|u|-1]}=v$ for all $n\in\N$. Since the word $v$ is just shifted as a sequence of admissible block of $1$ by the action of $T$, one deduces that $T^{i_n+k}(x)_{[0,|u|-1]}=u$ for all $n\in\N$, that is to say $\omega(x)\cap[u]_0\ne \emptyset$ so
 $$A\subset\omega(x).$$
 
 The previous claim gives the opposite inclusion, thus $\omega(x)=A$ for every $x\in \Lambda$.\eclaimprf

It remains to show that $A$ is also the statistical attractor of $T$. Let $\mu$ be a shift-ergodic measure on $X$. For $x\in X$ and $n\in\N$, we recall that 
$$
\nu_{n}(x)=\frac{1}{n} \sum_{i<n} \delta_{T^{i}(x)}.
$$ 

Define $\phi:X\longrightarrow X$ such that $\phi(x)_i\ne x_i$ if and only if there exist $j_{1}< i < j_{2}$ such that $x_{[j_{1},j_{2}]}=01^{l}0$ and the machine $M_{l}$ halts on the empty input. In other words the function $\phi$ keeps in $x$ only the blocks of 1's whose length correspond to an index of a Turing machine that does not halt on the empty input. The function $\phi$ is continuous on $X\setminus\left\{w1^{\infty}:w\in\{0,1\}^{\ast}\right\}$ so it is measurable. 

The function $\phi$ is not shift invariant since if the machine $l$ halts on the empty input, one has $\phi(\sigma(x))\ne\sigma(\phi(x))$ if $x_{[0,l+1]}=01^l0$. However, if a word $u$ starts by a $0$, $\phi^{-1}([u]_n)$ can be written as an union of cylinders which begin at position $n+1$ so $\sigma^{n-m}\phi^{-1}([u]_n)=\phi^{-1}([u]_m)$  for $n\leq m$. By $\sigma$-invariance of $\mu$, one has $\phi\mu([u]_n)=\phi\mu([u]_m)$. If $u$ does not begin by the symbol $0$, denote $B_i^j$ the set of configurations which have at least one $0$ between the coordinate $i$ and $j$. One has \[\sigma^{n-m}\phi^{-1}([u]_n\cap B_0^n)=\phi^{-1}([u]_m\cap B_{m-n}^m)\] so $\phi\mu([u]_n\cap B_0^n)=\phi\mu([u]_m\cap B_{m-n}^m)$. Moreover one has \[\phi\mu (B_{m-n}^m)\geq\mu(B_{m-n+1}^{m+1})=\mu(B_0^n)\underset{n\to\infty}{\longrightarrow}1,\] 
and thus it follows that the sequence  $(\sigma^n\phi\mu)_{n\in\N}$ converges towards a measure that we denote by $\tilde{\mu}$. Moreover if $\mu$ has full support, the support of $\tilde{\mu}$ is $A$.

\Claim{ 
Let $\mu$ be a shift ergodic probability measure of full support on $X$. Then $$\nu_{n}(x)\underset{n\to\infty}{\longrightarrow}\tilde{\mu}\textrm{ for $\mu$-almost every point $x$}.$$
}
\bclaimprf
Let $\epsilon>0$ and $u\in\{0,1\}^{\ast}$. Consider a number $l$ of a Turing machine which does not halt on the empty input.  

The set $\bigcup_{k\in\N}\sigma^{-k}([01^l0]_0)$ is $\sigma$-invariant. Since $\mu$ is $\sigma$-ergodic and has full support, this set has measure one so there exists $K$ such that $$\mu\left(\bigcup^K_{k=0}\sigma^{-k}([01^l0]_0)\right)\geq 1-\frac{\epsilon}{2}.$$

Consider
$$B=\left(\bigcup_{k=0}^K\sigma^{-k}([01^l0]_0)\right) \cap \sigma^{-(|u|+K+l+2)}\left(\bigcup_{k=0}^K\sigma^{-k}([01^l0]_0)\right).$$

One has $\mu(B)\geq 1-\epsilon$. Consider $x\in \sigma^{-n+K}(B)$. For all $n\geq K$, there exists $j_1,j_2\in\N$ such that 
\[n-K\leq j_1\leq j_1+l+1< n< n+|u|-1< j_2\leq n-K+2K+l+2+|u|\]
\[\textrm{ and }x_{[j_1,j_1+l+1]}=x_{[j_2,j_2+l+1]}=01^l0.\] 

Let $t$ be the biggest halting time obtained considerng Turing machine $M_{l'}$ with $l'\leq 2K+l+2+|u|$.   If $n$ satisfies also $n\geq t+K$, we have that all blocks of 1's contained in $x_{[j_1,j_2+l+1]}$ have length less than $2K+l+2+|u|$ so they are erased by the iteration of $T$ if the corresponding Turing machines halt, since the position $j_1$ is larger than $t$. It follows that 

$$T^n(x)_{[0,|u|-1]}=\phi(x)_{[n,n+|u|-1]}.$$

Thus, for all $x\in X$ and for $n\geq t+K$, one has \[\left|\delta_{T^n(x)}[u]-\mathbf{1}_{[u]}(\sigma^n\phi(x))\right|\leq1-\mathbf{1}_{\sigma^{-n+K}(B)}(x)\] so that

\[\left|\frac{1}{N}\sum_{n=0}^N\delta_{T^n(x)}[u]-\frac{1}{N}\sum_{n=0}^N\mathbf{1}_{[u]}(\sigma^n(\phi(x)))\right|\leq1-\frac{1}{N-K}\sum_{n=K}^N\mathbf{1}_{\sigma^{-n+K}(B)}(x)\]
where the last quantity converges to $1-\mu(B)=\epsilon$ as $N\to \infty$ for $\mu$-almost every $x$. 
We conclude that for $\mu$-almost every $x$ one has 
$$\frac{1}{N}\sum_{n=0}^N\delta_{T^n(x)}[u]\underset{N\to\infty}{\longrightarrow}\tilde{\mu}[u].$$

\eclaimprf

Since $A$ is the support of $\tilde{\mu}$, it follows that it is indeed the statistical attractor of $T$.

\end{proof}

\subsection{$\Pi_{2}$-complete wild attractors}\label{section.Pi2}

In this section we provide two kinds of examples. First, we construct an example of a system with a $\Pi_{2}$-complete invariant set $A$ which is the unique attractor in both the topological and metric sense. Note that according to the upper bounds of Section~\ref{section.UpperBounds}, this attractor cannot be strictly attracting nor can it be a statistical attractor. It is known that the topological and metric attractors of a system do not need to coincide. When they are different, they are sometimes called \emph{wild} attractors \cite{Milnor-attractor,wild96}. As a second kind of example, we adapt our construction to show that computable systems can have wild attractors which, in addition of being different as sets, their computational descriptions have extremely different complexity. 

\begin{theorem}\label{theorem.Pi2}
 There exists a computable map $T:X\to X$ with a $\Pi_{2}$-complete invariant subset $A$ which is the unique attractor of $T$ in both, the topological and metric sense. The reference measure for the latter can be taken to be any shift-ergodic measure of full support.  
 \end{theorem}
\begin{proof}
 Consider the set  $$\mathcal{T}ot = \left\{i \in \N : M_{i} \text{ is total}\right\}$$ of indices of machines that halt on every input.  It is well known that this is a  $\Pi_{2}$-complete set~\cite{Rogers}. 
 
Let $A\subset \{0,1\}^{\N}$ be  the set of configurations which contains the configuration $0^\infty$, or this configuration with at most one block of the form $01^i0$ with $i\in\mathcal{T}ot$ or configurations which begin by the symbol $1$ and can have at most one change to have $0$ after. In particular $A$ is a countable subshift and it is $\Pi_{2}$-complete since

\[
[01^{i}0]_0\cap A\ne\emptyset \iff i \in \mathcal{T}ot. 
\]

  Let $X=\{0,1,S\}^{\N}$ and  let $\mu$ be a shift ergodic probability measure of full support. We describe a computable map $T: X \to X$ having $A$ as the unique topological and metric attractor on which $T$ acts as the shift. The idea is that given a configuration $x=u_0Su_1Su_2Su_3\dots$ where $u_i\in\{0,1\}^*$ and the symbol $S$ appears at positions $(i_n)_{n\in\N}$, the computation of $T(x)$ is carried on according to the following steps: 
\begin{enumerate}
\item $u_0$ is shifted to the left;
\item the first $S$ \textit{travels to the right at speed one}, i.e., it moves to right one position at each iteration of $T$. While $u_0$ contains $1's$, $S$ pushes to the right $u_1$ together with the rest of the sequence, adding $0's$ to its left to fill the positions that otherwise would have been left with no symbols assigned. If there is no $1's$ in $u_0$, then $S$ copies to its left, symbol by symbol, the entire first block of $1's$ in $u_1$, if any, adding a $0$ at the end of $u_0$ in order to have a valid configuration (all positions with some symbol assigned). In other words, $S$ pushes everything to the right until there are only $0's$ to its left, and then lets the first block of $1's$ on its right to cross, one symbol at a time. Once the entire block has crossed, it then pushes everything to the right again until the block that has just crossed (which is now being shifted to the left) disappears. It then continues in this same manner. 
%\item if $u_0\in\{0\}^*$ and there is a block of $1's$ of length $l$ in the  $i_1$ first bits of $u_1$ such that $M_l$ halts on any input of size $1$ in at most $i_1$ steps, then this block of $1's$ is placed to the left of the $S$ at position $i_1$. If not, $Su_1$ is shifted to the right and the liberated positions are filled with the symbol $0$;  
\item for the $k^{\textrm{th}}$ symbol $S$ with $k\geq 2$, we search all blocks of $1$'s in the first $i_k$ bits of $u_k$ such that the corresponding Turing machines halt on any input of size less than $k$ in at most $i_k$ steps. All these blocks are placed at once on the left of this symbol $S$, so that $S$ ``jumps" to the right of these blocks and ``pushes" the others. 
\end{enumerate}

The function $T$ is clearly computable. Moreover, given a configuration $x$, all symbols $S$ in $x$ are shifted to the right by at least one position when $T$ is applied. 

Let $G\subset X$ be the set of configurations which contain infinitely many symbols $S$, and infinitely many blocks of $1$'s of all lengths. Clearly $G$ is a $G_{\delta}$ set of measure one. We now show that $\omega(x)=A$ for all $x\in G$. 

Consider $x\in G$. Since the first $S$ is shifted to the right and a block of $1$'s crosses the first $S$ only if there is no symbol $1$ before, it follows that $\omega(x)$ is contained in the subset of $X$ which contains at most one block of $1$'s. Moreover $T$ acts as the shift on $\omega(x)$.

Let $l\not\in\mathcal{T}ot$. There exists $k\in\N$ such that the Turing machine $M_l$ does not halt on an input of size $k$. Thus the $k^\textrm{th}$ symbol $S$ of $x$ cannot be crossed by a block of $1$'s of size $l$. It follows that the orbit of $x$ can intersect the cylinder $[01^l0]$ only finitely many times. 

Let $l\in\mathcal{T}ot$. For $k\in\N$, denote by $t_k$ the maximum halting time of $M_l$ over all inputs of size less than $k$. Now consider a block of $1$'s of size $l$, that is preceded by $k$ symbols $S$. Since the length of the word between this block and the $k^\textrm{th}$ symbol $S$ can only decrease, and since all symbols $S$ are shifted to the right, there exists an iteration of $T$ such that  the $k^\textrm{th}$ symbol $S$ is in a position $i\geq t_k$ such that the distance between the symbol $S$ and the end of the block of $1$'s is less than $i$. At this point the block can cross this symbol $S$, and go in a similar way through the remaining $k-1$ symbols $S$. Thus, the block $[01^l0]$ reaches position $0$ and then disappears infinitely many times and so that $\omega(x)\cap [01^l0]\neq \emptyset$.  It follows that that $\omega(x)=A$.

\end{proof}

%\begin{remark}
%It is possible to modify the previous construction in view that $T$ is transitive on $A$: the first symbol $S$ has a buffer which %memorizes all blocks of $1$ which cross it. This buffer is %shifted on the right but it generates a transitive configuration %with the blocks memorized. \todo{This remark seems false!}
%\end{remark}

It is possible to modify the previous construction in view to realizes maps where the topological and metric attractors have completely different complexity.

\begin{theorem}\label{Corollary.DifferentAttractors}
 There exists computable maps $T':X\longrightarrow X$ and $T'':X\longrightarrow X$ such that: 
\begin{itemize}
 \item $T'$ has a computable metric attractor and a $\Pi_2$-complete topological attractor.
 \item $T''$ has a  $\Pi_2$-complete  metric attractor and a computable topological attractor.
\end{itemize}
 \end{theorem}
\begin{proof}
Consider $X_1=\{0,1,S\}^{\N}$ and $X_2=\A^{\N}$ where $\A=\{a,b\}$. Define $X=X_1\times X_2$ and denote $\lambda$ be the uniform Bernoulli measure on $X$. For a sequence $x\in X$ we refer to the symbols in $X_1$ and $X_2$ as in the first and second layers, respectively. 

For $t\in\N$, on the space $X$, consider the dense open set $U_t$ defied by 
\[
U_t=\bigcup_{s\in\N} U_{s,t}
\quad
\textrm{ where }
\quad
U_{s,t}=\bigcup_{u_1\in\{0,1\}^{s-1}}
\quad
\smashoperator[r]{\bigcup_{
\substack{u_2\in\mathcal{A}^{\ast}, \\
\textrm{ with }|u_2|\geq t}} }
\quad [u_1S]_0\times [u_2a^{|u_2|+s}]_0.
\]

One has 
$$\lambda(U_{s,t})\leq\frac{1}{3}\times\left(\frac{2}{3}\right)^{s-1}\times\sum_{t'\geq t}\left(\frac{1}{2}\right)^{t'+s} =\left(\frac{2}{3}\right)^s\times\frac{1}{2^{t+s}}\leq\left(\frac{2}{3}\right)^s\times\frac{1}{2^{t}}.$$

So $\lambda(U_t)\leq \frac{3}{2^t}$ and clearly $G'=\bigcap_{k\in\N}U_{k}$ is a dense $G_\delta$ of measure $0$.

Consider $T':X\longrightarrow X$ be the computable map such that $T'$ acts on $X_2$ as the shift and acts on $X_1$ as the function defined in Theorem~\ref{theorem.Pi2} except for the first $S$: a block of $1's$ in the first layer can cross the first $S$ at position $i$ only if one of the following conditions hold: either the word $a^{2i}$ appears on the second layer at position $i$, or the word $a^{2i+1}$ appears on the second layer at the position $i+1$. 

Since the first $S$ travels to the right at speed one, if at time 0 it is at position $s$ in the initial configuration, at time $t$ it will be at position $s+t$. Thus, this first $S$ meets the word $a^{2(t+s)}$ (resp.  $a^{2(t+s)+1}$) at time $t$ at position position $t+s$ (resp. $t+s+1$) exactly when this word is initially at position $2t+s$ (resp. $2t+s+1$) in $x$. In other words, when $x\in U_{s,2t+s}$ (resp. $x\in U_{s,2t+s+1})$. See Figure~\ref{fig:attra} for an illustration.

\begin{figure}[h!]
    \centering
\begin{tikzpicture}[scale=0.55]

\foreach \i in{0,1,...,4}{
 \fill[red!15] (11-\i,\i)rectangle(18,\i+0.5);
 \fill[blue!15] (12-\i,\i+0.5)rectangle(18,\i+1);

\draw[gray,dashed] (17,\i)--(18,\i);
 \draw(\i+3.5,\i+0.5) node{$S$};

 }
\draw[gray,dashed] (17,5)--(18,5);
 \draw[gray] (0,0) grid (17,5);

 \draw(0,0.5) node[left]{$x=$};
 \draw(0,4.5) node[left]{$T'^4(x)=$};

\begin{tiny}
\draw(3.5,0)--(3.5,-0.2);
 \draw(3.5,-0.5) node{$s=3$};
 
  \draw(7.5,0)--(7.5,-0.2);
 \draw(7.5,-0.5) node{$s+t=7$};
 
   \draw(11.5,0)--(11.5,-0.2);
 \draw(11.5,-0.5) node{$s+2t=11$};

  \draw(18.3,0.75) node[right]{$\in U_{s,s+2t+1}=U_{3,11}$};
  \draw(18.3,0.25) node[right]{$\in U_{s,s+2t}=U_{3,12}$};
 \end{tiny}
 
\end{tikzpicture}
    \caption{An illustration of the action of $T'$ on a configuration $x$ following that it belongs to $U_{s,s+2t}$ or $U_{s,s+2t+1}$. The red rectangle corresponds to the word $a^{s+2t}$ and the blue rectangle to  the word $a^{s+2t+1}$  following that $x$ is $U_{s,s+2t}$ or $U_{s,s+2t+1}$.}
    \label{fig:attra}
\end{figure}
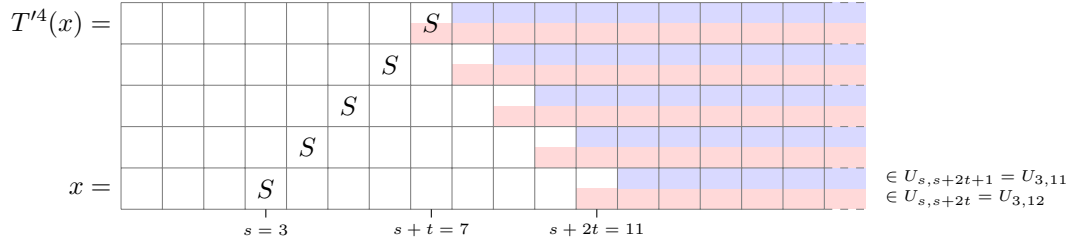

It follows that the first $S$ of a configuration $x$ allows a block of $1's$ to cross it infinitely often if and only if there exist $s\in\N$ and a strictly increasing sequence $(t_i)_i\in\N$ such that $x\in U_{s,t_i}$ for all $i\in\N$. That is, exactly when $x\in G'$. 

Now, let $G\subset G'$ be the dense $G_\delta$ set of measure $0$ made from  configurations of $G'$ for which the first layer contains an infinity of $S$ and an infinity of blocks of $1's$ of all sizes. For a configuration $x\in G$, the first $S$ is crossed infinitely often by blocks of $1's$. Thus, following the proof of Theorem~\ref{theorem.Pi2}, $\omega(x)=A\times \{a,b\}^{\N}$ where $A$ is the attractor of the function defined in the previous Theorem. It follows that $A\times \{a,b\}^{\N}$ is the topological attractor of $T'$. Now, if $x\notin G'$, there is only a finite number of  blocks of $1's$ that can cross the first $S$, so almost surely $\omega(x)=\{0^{\infty}\}\times\{a,b\}^{\N}$, which is therefore the metric attractor of $T'$. This proves the first bullet point of the corollary.

Now consider $T'':X\longrightarrow X$ defined as follows. It acts on $X_2$ as the shift and acts on $X_1$ as the function defined in the proof of Theorem~\ref{theorem.Pi2}, with the following modifications. The second $S$ in a configuration $x$ will now behave as in $T$, but in addition, it will check whether the first $S$ satisfies one of the following conditions: if $i$ is the position of the first $S$, then either the word $a^{2i}$ appears in the second layer at position $i$, or $a^{2i+1}$ appears in the second layer at position $i+1$. In this case, the function $T''$ inserts, starting at the position where the second $S$ is located, the concatenation of the first $i$ blocks of $1's$ (in lexicographic order), pushing to the right the rest of the configuration (including the second $S$) as many positions as needed in order to fit the inserted word.

Now, observe that for almost all $x\in X$, the conditions that the second $S$ check can be verified only finitely many times. Thus, $\omega(x)=A\times X_2$ for almost every $x\in X$, which makes $A\times X_2$ the metric attractor of $T''$. 

On the other hand, for $x\in G'\subset X$ every block of $1's$ will appear to the left of the second $S$ infinitely many times. Thus, $\omega(x)=A'\times X_2$ where  $A'\subset \{0,1\}^{\N}$ is  the set of configurations which contain at most one block of the form $01^i0$ with $i\in\N$. The set $A'$ is clearly computable, and therefore so is $A'\times X_2$, which is the topological attractor of $T''$. 

\end{proof}

\subsection{A $\Sigma_{2}$-complete statistical attractor}

\begin{theorem}
 There exists a computable map $T:X\to X$ with a $\Sigma_{2}$-complete statistical attractor.
\end{theorem}
\begin{proof}
 Consider the set  $$\mathcal{F}in = \{i \in \N : M_{i} \text{ has a finite domain}\}$$ of indices of machines that halt on a finite number of inputs.  It is well known that this is a  $\Sigma_{2}$-complete set. 
 
 Let $X=\{0,1\}^{\N}$, we describe a dynamics $T:X \to X$  having an invariant closed set $A$ which is the support of an invariant measure $\nu$ such that $\nu_{x} = \nu$ for almost every $x$ (with respect to a shift ergodic probability measure of full support). $A$ will therefore be the statistical attractor of $T$.
 
Define $A$ as the set of points $x\in X$ which do not contain any of the words of the form  $10^k1^{n}0$ for which the $n$-th Turing Machine halts on any input of size larger than $k$, nor any of the words of the form $01^n0$ for which $n\notin \mathcal{F}in$. This set is a closed shift invariant set. Moreover,  since
 $$
[01^{n}0]_0\cap A \ne\emptyset \iff n \in \mathcal{F}in,
$$
$A$ is $\Sigma_{2}$-complete.

The definition of $T$ is as follows. For $x\in X$, the $i^\textrm{th}$ bit of $T(x)$, denoted by $T(x)_{i}$, is given by:

\begin{enumerate}
	\item  if there exist $j_0<j_{1}\leq i < j_{2} \leq 2i$ such that $x_{[j_{1},j_{2}]}=01^{l}0$ with $l<j_{1}$,  $x_{[j_{0},j_{2}]}=10^k1^{l}0$ and the Machine $M_{l}$ halts on any input of size between $k$ and $j_1$ in at most $j_{1}$ steps, then $T(x)_{i}=0$.
		
	\item otherwise, $T(x)_{i}=x_{i+1}$. 
\end{enumerate}

In other words, the function $T$ deletes in $x$ every block of $1's$ of length $l$ that satisfies the following condition: if the block appears at position $i$ and is preceded by a block of $0's$ of size $k\leq i$, then the machine $M_{l}$ halts, in at most $i$ steps, on at least one input whose size is between $k$ and $i$. We remark that if a block is not deleted by one iteration of $T$, then it will be shifted so that it will eventually reach the position $0$. It is clear that such $T$ is a computable function.

Define $\phi':X\longrightarrow X$ such that $\phi'(x)_i\ne x_i$ if and only if there exist $j_{1}< i < j_{2}$ such that $x_{[j_{1},j_{2}]}=01^{n}0$ and the Machine $M_{n}$ has not a finite domain (that is to say $n\notin \mathcal{F}in$) or if $x_{[j_{1},j_{2}]}=10^k1^n0$ and the machine $M_{n}$ halts on an input of size larger than $k$. The function $\phi'$ is continuous on $X\setminus\left\{w1^{\infty}:w\in\{0,1\}^{\ast}\right\}$ so it is measurable. As in the proof of Theorem~\ref{thm:Pi_1-complete}, the function $\phi'$ is not shift invariant but if $\mu$ is shift invariant, the sequence  $(\sigma^n\phi'\mu)_{n\in\N}$ converges towards a measure denoted $\tilde{\mu}$. Moreover if $\mu$ has full support, the support of $\tilde{\mu}$ is $A$.

Let $\mu$ be a probability measure of full support on $X$ which is shift ergodic. We are going to show that $$T^n\mu\underset{n\to\infty}{\longrightarrow}\tilde{\mu}.$$

Let $\epsilon>0$ and $w\in\{0,1\}^{\ast}$. Consider two integers $k,l\in \N$ such that the Turing machine $M_l$ does not halt on any input of size larger than $k$. 

The set $\bigcup_{i\in\N}\sigma^{-i}([0^k1^l0]_0)$ is $\sigma$-invariant. Since $\mu$ is $\sigma$-ergodic and has full support, this set has measure one so there exists $K\geq k+l+1$ such that 
\[\mu\left(\bigcup^K_{i=0}\sigma^{-i}([0^k1^l0]_0)\right)\geq 1-\frac{\epsilon}{2}.\]

Consider
$$B=\left(\bigcup_{i=0}^K\sigma^{-i}([0^k1^l0]_0)\right) \cap \sigma^{-|w|+2K}\left(\bigcup_{i=0}^K\sigma^{-i}([0^k1^l0]_0)\right).$$

One has $\mu(B)\geq 1-\epsilon$.  Consider all pairs $(k',l')\in[0,3K+|w|]^2$ such that $M_{l'}$ halts on at least some input of size larger than $k'$ and denote by $w_{l',k'}$ the shortest input larger than $k'$ where $M_{l'}$ halts. Let now $N$ be the maximum between the biggest halting time on the inputs $w_{l',k'}$ and $|w_{l',k'}|$ for all these pairs. 

Consider $x\in\sigma^{-n}(B)$ for $n\geq K+N$, there exists $j_1,j_2\in\N$ such that 
\[n-K\leq j_1\leq j_1+k+l\leq j_1+K< n< n+|w|-1< j_2\leq n+2K+|u|\]
\[\textrm{ and }x_{[j_1,j_1+k+l]}=x_{[j_2,j_2+k+l]}=0^k1^l0.\]
 Thus all blocks of $1's$ contained in $x_{[j_1,j_2+k+l]}$ have length smaller than $2K+|w|$, so they are erased by $T$ if the corresponding Turing machines halts on at least some input of size larger than the number of $0$ that precedes them. As a block of $1's$ cannot disappear after one iteration of $T$, it follows that 
 \[T^n(x)_{[0,|w|-1]}=\phi'(x)_{[n,n+|w|-1]}\].
 
 Thus, for all $x\in X$, one has

$$\left|\delta_{T^n(x)}[w]-\mathbf{1}_{[w]}(\sigma\phi'^n(x))\right|\leq1-\mathbf{1}_{\sigma^{-n+K}(B)}(x),$$ it follows that for $\mu$-almost $x$, one has

\[\left|\frac{1}{N}\sum_{n=0}^N\delta_{T^n(x)}[w]-\frac{1}{N}\sum_{n=0}^N\mathbf{1}_{[w]}(\sigma^n(\phi'(x)))\right|\leq1-\frac{1}{N-K}\sum_{n=K}^N\mathbf{1}_{\sigma^{-n+K}(B)}(x)\underset{N\to\infty}{\longrightarrow}1-\mu(B)=\epsilon.\]

We conclude that for $\mu$-almost every $x$ it holds that
$$\frac{1}{N}\sum_{n=0}^N\delta_{T^n(x)}[w]\underset{N\to\infty}{\longrightarrow}\tilde{\mu}[w]$$

and that the support of $\tilde{\mu}$ is exactly $A$.

\end{proof}

\section{Interval maps with computationally complex attractors}\label{interval}

In this section we develop a simple construction allowing to embed any computable dynamics over a symbolic space into a computable dynamics of the unit interval  in a way that preserves metric and statistical attractors. We then use this construction to provide examples of computable maps of the interval exhibiting computationally complex attractors.   

%\subsection{Embedding  transformations over $\{0,1\}^{\N}$ into maps over $[0,1]$}

\begin{theorem}\label{theorem.EmbeddingCantorTransformationInIntervalMaps} Let $T:\{0,1\}^{\N}\to\{0,1\}^{\N}$ be a computable transformation. Then there exists a computable cantor set $\C\subset [0,1]$, a computable homeomorphism $\phi:\{0,1\}^{\N}\to \C$ and a computable map $f:[0,1] \to [0,1]$ preserving $\C$ such that:
\begin{enumerate}
\item\label{a}  $\phi$ conjugates $T$ to $f$ over $\C$,
\item\label{b}  for Lebesgue almost every $x\in [0,1]$, $\omega_{f}(x)\subset \C$,
\item\label{c} if $T$ has a metric or statistical attractor, then  $\mathcal{A} = \phi(A) \subset \C$ is an attractor for $f$ of the same type.  
\end{enumerate}
\end{theorem}

\proof 
Let $\Leb$ denote  Lesbesgue measure over $[0,1]$. We construct a fat Cantor set $\C$ with Lebesgue measure $1/2$. Of course the number $\frac{1}{2}$ could be replaced by any other positive number smaller than $1$. 
Let $a_n$ be a computable sequence of positive real numbers such that $\sum_n a_n = 1$ and let $b_n=2^{-a_n}<1$. Then, 
$$
\prod_n b_n=2^{-\sum_n a_n} = \frac{1}{2}.
$$ 

For every word $w\in\{0,1\}^*$, we define a closed interval $I_w \subset [0,1]$ inductively on $|w|$.  Start by setting $I_\epsilon=[0,1]$, where $\epsilon$ is the empty word.  Assume $I_w$ has been defined. Then,  $I_{w0}$ and $I_{w1}$ are constructed by removing an open segment centred at the middle of $I_w$, such that 
$$
\Leb(I_{w0})+\Leb(I_{w1})=b_{|w|}\Leb(I_w).
$$  

We denote by 
$$
\C_{n} = \bigcup_{|w|=n} I_{w}
$$ 
the collection of the $2^{n}$ intervals constructed at stage $n$.  The cantor set $C$ is then defined by
$$
\C=\bigcap_n \C_{n}. 
$$
Note that, by construction, 
$$
\Leb(\C_{n+1}) = \sum_{|w|=n}b_{n}\Leb(I_{w}) = b_{n}\Leb(\C_{n})
$$ 
and therefore 
$$
\Leb(\C)=\prod_n b_n=\frac{1}{2}.
$$
The extreme points of each $I_w$ are computable real numbers, uniformly in $w$. Thus $\C$ is  a computable set. 

For any given finite binary sequence $w$, an easy computation shows that
$$
\Leb(I_{w}\cap \C_{|w|+k})=2^{-|w|} \prod_{i\ < |w|+k }b_{i}
$$
and thus we have that the relation
\begin{equation}\label{restriction}
\Leb(I_{w}\cap \C)=2^{-|w|} \Leb (\C)
\end{equation}
holds. 

Let $\phi:\{0,1\}^\N\to \C$ be the homeomorphism defined by 
$$
\phi(x)=\bigcap_{n}I_{x_{[0,n]}}. 
$$

It is clear that $\phi$ is computable. Note that by equation (\ref{restriction}), $\phi$ sends the uniform measure $\uni$ over $\{0,1\}^{\N}$ to the  Lesbesgue measure $\Leb$ conditioned to $\C$:  
$$
\Leb(\cdot|\C)=\frac{\Leb(\cdot\cap \C)}{\Leb(\C)}. 
$$ 

Moreover, since 
\begin{equation}\label{interval.measure}
\Leb(I_{w})=2^{-|w|}\Leb(\C_{|w|})=2^{-|w|}\prod_{i< n}b_{i}<2^{-|w|},
\end{equation}
 we see that $d(x,y)\leq 2^{-n}$ implies $|\phi(x)-\phi(y)|\leq 2^{-n}$, where $d(x,y)$ denotes the usual distance on binary sequences. In the oposite direction, however, observe that even if $|\phi(x)-\phi(y)|\leq 2^{-(n+1)}$, $\phi(x)$ and $\phi(y)$ may lie in different segments of $\C_{n}$ and thus $d(x,y)$ could be greater than $2^{-n}$.  On the other hand, since the smallest gaps between segments of $\C_{n}$\footnote{that is, between $I_{w0}$ and $I_{w1}$, where $|w|=n-1$} have length 
$$
2^{-(n-1)}  (1-b_{n-1}) \prod_{i < n-1} b_{i},  
$$
we have that 
\begin{equation}\label{distortion}
|\phi(x)-\phi(y)| < 2^{-n}(1-b_{n-1}) \implies d(x,y)\leq 2^{-n}. 
\end{equation}
Now, for $\alpha \in \C$, define $f$ by 
$$
f(\alpha)=\phi \circ T \circ \phi^{-1} (\alpha)
$$
so as to make $f$ conjugate to $T$ on $\C$.  We now explain how to define $f$ on the gaps.  Since $T$ is computable and $\{0,1\}^{\N}$ is recursively compact, it has a computable modulus of uniform continuity. In particular, we can compute a function $m:\N\to\N$ satisfying 
$$
m(n) \nearrow  \infty  \quad \text{ as } \quad n \to \infty 
$$
and such that
$$
d(T(x),T(y))\leq 2^{-m(n)}\quad \text{ whenever }\quad d(x,y)\leq2^{-n}.
$$
Let $w$ be a finite binary word of length $n$ defining a segment  $I_{w}$ of $\C_{n}$ and let
$$
a=\phi(w01^{\infty}); \qquad b=\phi(w10^{\infty}) 
$$
be the extreme points of the gap between the two consecutive segments $I_{w0}$ and $I_{w1}$ of $\C_{n+1}$.  Then
$$
d(w01^{\infty},w10^{\infty})\leq 2^{-n}
$$ 
and thus 
$$
d(T(w01^{\infty}),T(w10^{\infty})) \leq 2^{-m(n)}.  
$$
It follows that $f(a)=\phi(T(w01^{\infty}))$ and $f(b)=\phi(T(w10^{\infty}))$ both belong to a same segment $I'$ of $\C_{m(n)}$. We then define $f$ on $[a,b]$ in a piecewise linear and continuous way so as to have 
$$
f([a,b])=I'
$$
with at most 3 pieces in such a way that at least half of $[a,b]$ gets linearly mapped to $I'$ (see Figure~\ref{fig:my_label}).

\begin{figure}
    \centering
\begin{tikzpicture}[scale=0.7]

\begin{tiny}

\draw (3,2)--(4,1.5)--(5,4)--(6,3);

\draw (3,2) node{$\bullet$};
\draw (6,3) node{$\bullet$};

\draw[dotted] (0,4)--(8,4);

\draw[dotted]  (0,1.5)--(8,1.5);

\draw[->,>=latex] (0,-0.5)--(0,5);
\draw[->,>=latex] (-0.5,0)--(10,0);
 
\draw (3,0)--(3,-0.2);
\draw (3,-0.2) node[below]{$a$};
\draw (6,0)--(6,-0.2);
\draw (6,-0.2) node[below]{$b$};

\draw (0,2)--(-0.2,2);
\draw (-0.2,2) node[left]{$f(a)$};
\draw (0,3)--(-0.2,3);
\draw (-0.2,3) node[left]{$f(b)$};

\draw[{[-]}] (1.5,-0.8)--(3,-0.8);
\draw (2.25,-0.8) node[below]{$I_{w0}$};

\draw[{[-]}] (6,-0.8)--(7.5,-0.8);
\draw (6.75,-0.8) node[below]{$I_{w1}$};
 
\draw[{[-]}] (-1.4,1.5)--(-1.4,4);
\draw (-1.4,2.75) node[left]{$I'$};
 
\end{tiny}
\end{tikzpicture}

    \caption{The function $f$ on the gap $[a,b]$ between intervals $I_{w0}$ and $I_{w1}$ of $C$.}
    \label{fig:my_label}
\end{figure}
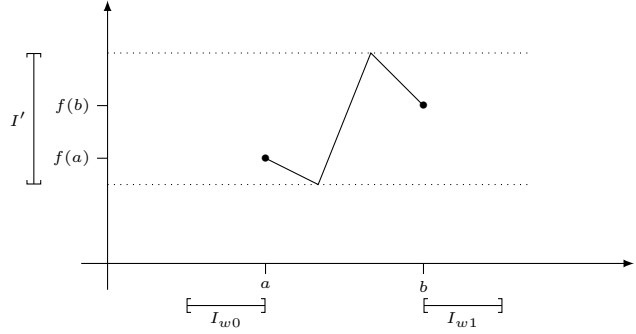

Note that $f$ is well defined and continuous. In particular, by equation (\ref{distortion}) we have that for any $x,y$ in $[0,1]$:
$$
|f(x)-f(y)| \leq 2^{-m(n)} \quad   \text{ whenever } \quad |x-y|\leq 2^{-s(n)}
$$
where $s(n):\N\to\N$ is an increasing function such that $$2^{-s(n)}\leq 2^{-n}(1-b_{n-1}).$$  This shows part (\ref{a}) of the Theorem.

To prove part (\ref{b}) note that, by construction, the proportion of points in the complement of $\C$ that fall in $\C$ after one iteration of $f$ is at least a fourth.  Indeed, since $f$ sends more than half of any given gap $[a,b]$ onto a whole segment $I^{k}$ of a certain level $\C_{k}$ in a linear way, it follows that 
$$
\Leb\left([a,b] \bigcap f^{-1}(\C) \right) = \Leb\left([a,b] \bigcap f^{-1}(\C \cap I^{k}) \right) \geq \frac{\Leb([a,b])}{2} \frac{\Leb(\C\cap I^{k})}{\Leb(I^{k})}$$
which by equations (\ref{restriction}) and (\ref{interval.measure}) equals 
\begin{equation}\label{eq:proportion}
\frac{\Leb([a,b])2^{-k}\Leb(\C)}{2^{-k-1}\Leb(\C_{k})}\geq\frac{\Leb[a,b]}{4},
\end{equation}
so we obtain
\begin{equation}
\Leb\left( ([0,1]\setminus\C)\cap f^{-1}\C\right) \geq \frac{\Leb([0,1]\setminus\C)}{4}.
\end{equation}
Since $\C$ is invariant under $f$, it follows that the Lesbesgue measure of the set of points that remain in the complement of $\C$ after $n$ iterations of $f$ is at most $\left(\frac{3}{4}\right)^{n}$, and thus $\omega(x)\subset \C$ for almost every $x\in[0,1]$. 

To prove (\ref{c}), recall that $\phi^{-1}$ sends $\Leb(\cdot | \C)$ to $\uni$, and thus if $\uni(\phi^{-1}E)=0$, then $\Leb(E)=0$. Moreover, $f$ is clearly a non-singular map for $\Leb$:
$$
\Leb(E)=0 \implies \Leb(f^{-n}E)=0  \qquad  \text{ for all }\quad n.
$$
Thus, no set other than $\A$ can have a basin of attraction of positive measure, which is therefore the metric attractor for $f$ whenever $A$ is the metric attractor for $T$.

Finally, we show that if $T$ has a statistical attractor $A$, then $\mathcal{A}=\phi(A)$ is a statistical attractor for $f$. Let $\tilde{\mu}$ be the unique measure whose support is $A$ and consider over $[0,1]$ the push forward measure $\mu =\tilde{\mu}\circ  \phi^{-1} $. It is clear that for almost every point  $x\in\C$, we have $\nu_x=\mu$. Suppose there exists a positive measure set $E\subset [0,1]\setminus \C$ such that $\nu_x\neq\mu$ for all $x\in E$. Let $E_1=\{x\in E: f(x)\in \C \}$ be the set of points in $E$ that fall in $\C$ after one iteration of $f$. Note that by inequality (\ref{eq:proportion}), $\lambda(E_1)>0$. But since $f$ is non-singular, $f(E_1)$ would be a positive measure set in $\C$ whose elements satisfy $\nu_x\neq \mu$, a contradiction.

\endproof

%\subsection{Consequence for the worst complexity for the different notion of attractors}

\begin{corollary}
There exist a computable map $f:[0,1]\to [0,1]$ and an invariant set $\mathcal{A}\subset [0,1]$ such that:
\begin{enumerate}
    \item $\mathcal{A}$ is $\Pi_1$-complete;
    \item $\mathcal{A}$ is a transitive  metric and statistical attractor for $f$;
\end{enumerate}  
In particular, $f$ has a unique physical measure which is not computable.
\end{corollary}

\begin{proof}
Apply Theorem \ref{theorem.EmbeddingCantorTransformationInIntervalMaps} to the construction of Theorem \ref{thm:Pi_1-complete}. 
\end{proof}

\begin{remark}
We note that the map $f:[0,1]\to[0,1]$ given by the above Corollary has a $G_\delta$-dense set of points in $[0,1]$ that never enter the cantor set $\C$ under iteration by $f$.  Therefore, the attractor $\mathcal{A}$ of $f$ is not a topological attractor, nor it is strongly attracting, despite that so is the set $A$ for $T$. 
\end{remark}

\begin{corollary}
There exist a computable map $f:[0,1]\to [0,1]$ with a $\Pi_2$-complete metric attractor $\mathcal{A}$.
\end{corollary}
\begin{proof}
Apply Theorem \ref{theorem.EmbeddingCantorTransformationInIntervalMaps} to the construction of Theorem \ref{theorem.Pi2}. 
\end{proof}

\begin{remark}
We note that by Theorem \ref{upper_bounds}, the set $\mathcal{A}$ cannot be a statistical attractor. Indeed, in virtue of the properties of the map $T$ constructed in Theorem \ref{theorem.Pi2}, the statistical attractor of $f$ is the singleton $\{0\}$. 
\end{remark}

\end{document}